\documentclass[12pt,leqno]{article}
\catcode`\@=11 \long\def\@makecaption#1#2{
 \vskip 10pt
 \setbox\@tempboxa\hbox{\bf #1: \sf #2}
 \ifdim \wd\@tempboxa >\hsize \bf #1: \sf #2\par \else \hbox
to\hsize{\hfil\box\@tempboxa\hfil}
 \fi}
\catcode`\@=12
\usepackage{amsmath}
\usepackage{amsfonts}
\usepackage{amssymb}
\usepackage{epsfig}
\def\qed{\kern 6pt\hbox{\vrule\vbox to 6pt{\hrule width
                            6pt\vfil\hrule}\vrule}}

\newcommand{\R}{\mathbb R}

\newcommand{\conv}{{\rm conv}}

\newtheorem{theo}{Theorem}%[section]
\newtheorem{lem}[theo]{Lemma}

\newenvironment{proof}[1][Proof]{\textbf{#1.} }{}%{\mbox{}\hfill \qed}

\def\conv{\mathrm{conv}}

\def\inte{\mathrm{int}}
\def\St{\mathrm{St}}

\begin{document}
\title{ Ellipsoids are the only local maximizers of the volume product
\footnotetext{2010 Mathematics Subject Classification 52A20, 52A40.}
}
%\footnote{2010 Mathematics Subject Classification. Primary  52A20. }
\author{{ Mathieu Meyer and Shlomo Reisner}}
\date{}

\maketitle

\vskip 10mm
\begin{abstract} Using previous results about shadow systems and Steiner symmetrization, we prove that the local maximizers of
the volume product of convex bodies are actually the global
maximizers, that  is: ellipsoids.

\end{abstract}

%%%%%%%%%%%%%%%%%%%%%%%%%%%%%%%%%%%%%%%%%%%%%%%%%%%%%%%%%%%%%%%%%%
\vskip 10mm Let $K\subset \R^n$ be a convex body (a compact and
convex set with non-empty interior). For $z\in\inte(K)$, the
interior of $K$, let $K^z$ be the polar of $K$ with respect to
$z$:
$$K^z=\{y\in  \R^n; \langle y-z,x-z\rangle \le 1\hbox{ for every  $x\in K$}\},$$
where $\langle.\,,\,.\rangle$ denotes the standard scalar product in
$\R^n$. It is well known that $K^z$ is also a convex body,
that $z\in \inte(K^z)$  and that $(K^z)^z=K$. The {\em volume
product\/} of $K$, $\Pi(K)$ (or $\Pi_n(K)$ if the dimension is to be specified), is given by the
following formula:
$$\Pi(K):=\min_{z\in \inte(K)} |K| \ |K^z|\,,$$
where $|A|$ denotes the Lebesgue  measure of a Borel subset $A$ of $\R^n$. The unique point $z= s(K)\in K$,
where this minimum is reached, is called the Santal\'o point of $K$. We denote  $K^*= K^{s(K)}$.
Blaschke \cite {B} (1917) proved for dimensions $n=2$ and $n=3$ that
$$\Pi(K)= |K| \ |K^*|\le \Pi(B_2^n)\,,$$
where $B_2^n=\{x\in \R^n; |x|\leq 1\}$ ($|x|=\sqrt{\langle x,x\rangle}$) is the Euclidean unit ball in $\R^n$.
This was generalized to all dimensions by Santal\'o \cite{San} (1948).
\vspace{2mm}

It  then took some time to establish the case of equality: one has $\Pi(K)=\Pi(B_2^n)$ if and only if $K$
is an ellipsoid.
This was done by Saint-Raymond \cite{Sai} (1981), when $K$ is centrally symmetric and by Petty \cite{P} (1982),
in the general case. Another proof was given by Meyer and Pajor \cite{MP} (1990),
based on Steiner symmetrization.

Campi  and Gronchi \cite{CG} (2006), introduced the use of shadow systems for volume product problems.
Fix a direction $u\in S^{n-1}$ . A shadow  system $(K_t)$ along the direction $u$ is a family of convex sets $(K_t)$, $t\in [a,b]$
such that
$$K_t=\conv\{x+t\alpha(x) u\,; x\in A\}$$
where $A$ is a given bounded subset of $\R^n$ and $\alpha:A\to \R$ ia a given bounded function, called the {\em speed}
of the shadow system.
An example is given by the {\em Steiner symmetrization\/} of a convex body $K$ with respect to the hyperplane $u^{\perp} $ orthogonal
to $u\in S^{n-1}$. If $K$ is described  as
$$K=\{y+su\,; y\in P_u K\,, s\in I(y)\}\,,$$
where $P_u $ is the orthogonal projection onto $u^{\perp} $ and $I(y)$ is some nonempty closed interval depending
on $y\in P_u K$.  The Steiner symmetral $\St_u(K)$ is defined by
$$\St_u(K):=\left\{y+su\,; y\in P_u K, s\in \frac{I(y)-I(y)}{2}\right\}\,.$$

For $t\in [-1,1]$, let
$$K_t=\left\{y+su\,; y\in P_u K\,,\, s\in \frac{1-t}{2}I(y)-\frac{1+t}{2}I(y)\right\}\,.$$

The family $(K_t)$, $t\in [-1,1]$ forms a shadow
system such that $K_{-1}=K$, $K_{1}$ is the reflection of $K$ with respect to $u^{\perp}$ and $K_0 $ is
the Steiner symmetral of $K$ with respect to $u^{\perp}$. As a matter of fact,
setting $A=K_0$, and  $I(y)=[a(y),b(y)]$ for $y\in P_u K$, one has for $t\in [-1,1]$:
$$K_t=\left\{z-t\frac{a(P_u z)+b(P_u z)}{2}u\,; z\in K_0\right\}\,.$$
\vskip 3mm

The following theorem was proved in \cite{MR2} as Theorem~1 and  Proposition~7 there.
\begin{theo}
\label{th-1}
Let $K_t$, $t\in [a,b]$, be a shadow system in $\R^n$. Then $t\to |K_t^*|^{-1}$ is a convex function
on $[a,b]$. If $t\to |K_t|$ and $t\to |K_t^*|^{-1}$ are both affine functions in $[a,b]$ then, for all
$t\in [a,b]$, $K_t$ is an affine image of $K_a$, $K_t=A_{u,t}(K_a)$. Where $A_{u,t}$ is an affine
transformation that satisfies $P_u\,A_{u,t}=P_u$. More precisely: for some $v\in \R^n$ and some $c\in \R$,
one has for all $t\in [-1,1]$ and all $x\in \R^n$:
$$A_{u,t}(x)=x+(t-a)\big(\langle x,v \rangle +c\big)u\,.$$
\end{theo}
\vskip 1mm
This theorem was extending  and strengthening a result of Campi and Gronchi \cite{CG}, who proved the
first part  of it when the shadow system $(K_t)$ is composed of bodies that are centrally symmetric with
respect to the same center of symmetry.
\vspace{1mm}

As a consequence of Theorem~\ref{th-1}, one gets the main result of this paper:

\begin{theo}
\label{th-2} The convex bodies $K$ in $\R^n$  which are  local maximizers (with
respect  to the Hausdorff distance or to the Banach Mazur distance) of
 the volume product in $\R^n$  are the ellipsoids.
\end{theo}
\vspace{2mm}

\noindent
{\bf Remark.}
 A partial result in this direction was proved by Alexander, Fradelizi and Zvavich \cite{AFZ} who observed that no polytope
can be a local maximizer for the volume product.
\vspace{3mm}

\noindent
\begin{proof}[Proof of Theorem~\ref{th-2}]
Suppose that $K$ is a local maximizer. Let $u\in S^{n-1}$ and
${\mathrm St}_u(K)$ be the Steiner symmetral of $K$ with respect
to $u^{\perp}$.

With the above notations we describe  the Steiner symmetral of $K$ as $K_0$ of  a shadow system $K_t$,
 $t\in [-1, 1]$, with  $K_{-1}=K$ and $K_{1}$ being the mirror
reflection of $K$ about  $u^{\perp}$.
It follows from the definition of this shadow system that it preserves the volume of $K$: one has $|K_t|=|K|$ for all
$t\in [-1,1]$.

By construction, for all $t$,
 $K_t$ is the mirror reflection of $K_{-t}$ with respect to $u^{\perp}$. It follows that  $(K_t)^*$ is also the mirror
 reflection of $(K_{-t})^*$ with respect to $u^{\perp}$
 Let $$f(t)= (|K| \ |(K_t)^*|)^{-1}= \frac{1}{\Pi_n(K_t)}\,.$$

 It is clear that the function $t\to K_t$ is continuous for both
 the Hausdorff and the Banach-Mazur distances. Thus such is also
 the function $t\to (K_t)^*$. It follows that $f$ is continuous on $[-1,1]$.

 By theorem \ref{th-1}, $f$ is convex on $[-1,1]$ and by construction,  it is even.
 Thus  $f(t)\le f(-1)=f(1)$  for all $t\in [-1,1]$ and $f$ has its absolute minimum at $0$.
  Since $K$ is a local maximum of the volume product (i.e, a local minimum of $f$), one has
 for some $-1<c\le 0$ , $f(t)\ge f(-1)$ for all $t\in [-1,c]$. Thus $f$ is constant on $[-1,c]$.
 It now follows from its convexity and the preceding observations,
 that $f$ is actually constant on $[-1,1]$ and
$|(K_t)^*|=|K^*|$ for $t\in [-1, 1]$.

From the second part of theorem \ref{th-1} we conclude now that $K_0=\St_u(K)$ is an image  of $K_{-1}=K$
under an affine transformation having special properties.
Since this fact is true for any $u\in S^{n-1}$, application of the next lemma completes the proof.

\begin{lem}\label{lemma 1}
Let $K$ be a convex body such that, for all $u\in S^{n-1}$, $\St_u(K)$ is an image of  $K$,
$St_u(K)=A_u(K)$ where $A_u$ is an affine transformation that satisfies $P_u\,A_u=P_u$. Then
(and only then) $K$ is an ellipsoid.
\end{lem}

\noindent

\noindent
{\bf Remark.\ }
Lemma~\ref{lemma 1} can be formulated in an equivalent form as: {\it Let $K$ be a convex body such that, for all $u\in S^{n-1}$,
the centers of the chords of $K$ that are parallel to $u$ are located on a hyperplane. Then (and only then) $K$ is
an ellipsoid.\/}
With this formulation the result, in dimension 2, was declared by  Bertrand \cite{Ber} (1842). But his proof does not seem
complete. The result was proved by Brunn \cite{Br} (1889). Gruber \cite{Gr} (1974) proved the result under strongly
relaxed assumptions. A number of proofs of the result appear in the literature. See e.g. Danzer, Laugwitz and Lenz \cite{DLL} (1957),
that use the L\"owner ellipsoid of $K$, or Grinberg \cite{Gri} (1991) that uses an infinite sequence of symmetrizations.
We bring here, for the sake of completeness, a  proof that uses the uniqueness of the John ellipsoid of $K$.

We also point out \cite{MR1}
for a generalization, replacing the location of midpoints of chords by the location of centroids of sections
of any fixed dimension $k$, $1\leq k\leq n-1$.

\vspace{2mm}

\noindent
\begin{proof}[Proof of Lemma~\ref{lemma 1}] We notice that the property of $K$ presented in the lemma is preserved
under affine transformations (this is easy to see from the equivalent form of this property presented in the Remark
above). Thus, using an affine transformation, we may assume that John's ellipsoid of $K$ (the ellipsoid of maximal
volume contained in $K$) is the Euclidean unit ball $B_2^n$. We then want to show that $K$ is a homothetic Euclidean ball.

Let $u\in S^{n-1}$. By the assumption, $St_u(K)=A_u(K)$, $A_u$ affine with $P_u\,A_u=P_u$. Hence the John ellipsoid of $St_u(K)$
is $A_u(B_2^n)$. Now $|K|=|St_u(K)|=|A_u(K)|$, so $|\det(A_u)|=1$ and $|A_u(B_2^n)|=|B_2^n|$. By symmetry of $B_2^n$ about $u^{\perp}$
and the fact that $B_2^n\subset K$, we have $B_2^n\subset St_u(K)$. By the uniqueness of the John ellipsoid we conclude
that $A_u(B_2^n)=B_2^n$. Thus $A_u$ is a linear isometry with respect to the Euclidean norm, i.e. an orthogonal transformation.

The orthogonal transformation $A_u$ preserves $u^{\perp}$ by the assumption of the lemma, so it is either the identity or an
orthogonal
reflection by $u^{\perp}$. Using any of these possibilities for each $u\in S^{n-1}$, we see that $K$ is orthogonally symmetric about any
hyperplane through $0$. It follows that all the points of the boundary of $K$ have the same Euclidean norm.
Thus $K$ is a Euclidean ball centered at the origin.

This completes the proof of Lemma~\ref{lemma 1}, thus also the proof of Theorem~\ref{th-2}. \hfill \qed

\end{proof}

\end{proof}

\vskip 10mm
\small{

\noindent M. Meyer: Universit\'e Paris-Est - Marne-la-Vall\'ee, Laboratoire d'Analyse et de
Math\'ematiques Appliqu\'ees (UMR 8050), Cit\'e Descartes, 5 Bd Descartes, Champs-sur-Marne,
77454 Marne-la-Vall\'ee cedex 2, France.\newline
Email: mathieu.meyer@u-pem.fr
\vspace{2mm}

\noindent S. Reisner: Department of Mathematics, University of
Haifa, Haifa 31905, Israel.
\newline
Email: reisner@math.haifa.ac.il
}

\end{document}